\documentclass[12pt]{amsart}

\usepackage{amssymb, latexsym}
\usepackage{float}

\def\mr{\mathrm}

\def\mbR{\mathbb{R}}
\def\mbZ{\mathbb{Z}}

\def\mbD{\mathbb{D}}

\newcommand{\mc}{\mathcal}

\newcommand{\mcR}{\mathcal{R}}

\theoremstyle{plain}
\newtheorem{theorem}{Theorem}[section]
\newtheorem{corollary}[theorem]{Corollary}
\newtheorem{proposition}[theorem]{Proposition}
\newtheorem{lemma}[theorem]{Lemma}

\theoremstyle{definition}
\newtheorem{definition}[theorem]{Definition}
\newtheorem{remark}[theorem]{Remark}
\newtheorem{example}[theorem]{Example}

\numberwithin{equation}{section}

\usepackage
{hyperref}

\usepackage[dvipsnames]{color}

\begin{document}

\textcolor[rgb]{0,0,1}{
}

\title[Geometry of dyadic polygons I] {Geometry of dyadic polygons I:\\ The structure of dyadic triangles}

\author[Mu\'{c}ka]{A. Mu\'{c}ka$^1$}
\address{$^1$ Faculty of Mathematics and Information Sciences\\
Warsaw University of Technology\\
00-661 Warsaw, Poland}

\author[Romanowska]{A.B. Romanowska$^2$}
\address{$^2$ Faculty of Mathematics and Information Sciences\\
Warsaw University of Technology\\
00-661 Warsaw, Poland}

\email{$^1$Anna.Mucka@pw.edu.pl\phantom{,}}
\email{$^2$Anna.Romanowska@pw.edu.pl\phantom{,}}



\date{September, 2025}

\begin{abstract}

Dyadic rationals are rationals whose denominator is a power of $2$.  \emph{A dyadic
$n$-dimensional convex set} is defined as the intersection with $n$-dimensional dyadic
space of an $n$-dimensional real convex set.   Such a dyadic convex set is said to be
a \emph{dyadic $n$-dimensional polytope} if the real convex set is a polytope whose vertices
lie in the dyadic space.
Dyadic convex sets are described as subalgebras of reducts of faithful affine
spaces over the ring of dyadic numbers, or equivalently as commutative, entropic and
idempotent groupoids (binars or magmas) under the binary operation of arithmetic mean.

This paper investigates the structure of dyadic polygons (two-dimensional polytopes), in particular dyadic triangles, following earlier results of~\cite{MRS11} and related papers. A new classification of dyadic triangles is provided. In addition, dyadic triangles with a pointed vertex are characterized.

\end{abstract}

\maketitle


\section{Introduction}

This paper is devoted to the study of dyadic polygons: mathematical objects that belong to geometry as subsets of certain affine spaces and to algebra as sets with one binary  idempotent, commutative and entropic operation. It may be considered as a continuation of earlier work from \cite{MRS11}.

Recall that dyadic rationals are rationals whose denominator is a power of $2$. They form the
principal ideal subdomain $\mbD = \mbZ[1/2]$ of the ring $\mbR$ of real numbers. The affine
spaces of interest are idempotent reducts of faithful $\mbD$-modules (modules over
$\mbD$). They are called \emph{affine $\mbD$-spaces} and are subreducts (subalgebras of reducts) of affine $\mbR$-spaces, i.e. affine spaces over the ring $\mbR$ of real numbers.

A subset $D$ of $\mbD^n$ is called a \emph{geometric dyadic convex set}
if it is the intersection of a convex subset of $\mbR^n$ with the space $\mbD^n$. Such sets are also called \emph{convex relative to $\mbD$} (cf. for instance \cite{B05}.) \emph{Dyadic polytopes} are dyadic convex sets obtained as the intersection of dyadic spaces and real (convex) polytopes whose vertices are contained in the dyadic space. \emph{Dyadic polygons} are the intersection of the dyadic
plane and real polygons with vertices in the dyadic plane. The one-dimensional analogs are
\emph{dyadic intervals} (considered as intervals with their ends).

Real convex sets are described algebraically as certain barycentric algebras, subsets of
$\mbR^n$ closed under the operations of weighted means with weights from the open unit interval
$I^{\circ}=\ ]0,1[$. See e.g. \cite{RS85, RS02}. Similarly, dyadic convex sets may be described as subsets of $\mbD^n$
closed under weighted means with weights from the open dyadic unit interval $I^{\circ} \cap
\mbD$. However, as shown already by Je\v{z}ek and Kepka \cite{JK76}, the
operations determined by $I^{\circ} \cap \mbD$ are all generated by the single arithmetic mean
operation
\[
x \circ y := xy \underline{1/2} = \frac{1}{2}(x + y).
\]
The arithmetic mean is a commutative, idempotent and entropic operation.
This fact allows one to define dyadic convex sets equivalently as algebras with one basic binary operation (groupoids, binars or magmas) with the algebraic structure of so-called \emph{commutative binary modes}
($\mc{CB}$-modes). Such algebras have a well developed theory. (See e.g. \cite{JK75, JK76, JK83, MMR19, MMR19a,
MMR23, MR04, MRS11, MR16, RS85, RS02}).

While all (closed) real intervals are isomorphic as ba\-ry\-cen\-tric algebras, and each is generated by
its ends, there are infinitely many pairwise non-isomorphic dyadic intervals (considered as
$\mc{CB}$-modes), classified first in~\cite{MRS11}. In particular,
each non-trivial interval of $\mbD$ is isomorphic to some dyadic interval $\mbD_k = [0,k]$,
where $k$ is an odd positive integer. (See also Section~\ref{S:Triangles}.)

Dyadic triangles were characterised in~\cite{MRS11}, with some
improvements given in~\cite{MMR19a} and~\cite{MMR23}. Roughly speaking, each dyadic triangle of the plane $\mbD^2$ is isomorphic to one of three types of triangles located in the first quadrant of the plane, with vertices having integral coordinates, and one (pointed) vertex located at the origin. (See Theorem~\ref{T:classif} of Section~\ref{S:Triangles}.) The description provided in Theorem~\ref{T:classif} depends on the choice of a vertex at the origin, which was not pointed out clearly enough in the original paper.
This fact,  and a mistake in Lemma~4.11 of~\cite{MRS11}, resulted in an incorrect formulation of one of consequences of this classification, Theorem~6.6, concerning characterization of dyadic triangles by means of certain quadruples of integers.
In this paper we provide a much simpler and clearer description of dyadic triangles with one pointed vertex. We show that each is isomorphic to a certain special dyadic triangle we call \emph{representative}. In consequence, we get a classification of all dyadic triangles with a pointed vertex, and their characterization by means of certain triples of integers. These results essentially improve and correct the earlier results. A subsequent paper will deal with isomorphisms of pointed representative triangles.   Because each dyadic polygon can be decomposed as a union of dyadic triangles whose interiors do not intersect, a good description of dyadic triangles is essential to an understanding of the structure of all dyadic polygons.
Since each dyadic polygon is isomorphic to a polygon with vertices having integral coordinates, its real convex hull is a \emph{lattice polygon} (cf. \cite{PRS}, \cite{GH}). Thus, a classification of dyadic convex polygons brings new insights to the study of lattice polygons.

This introductory work on triangles and polygons is intended as a preparation for a more extensive study of general dyadic polytopes.

Dyadic geometry is more than
a fascinating field of study from the mathematical point of view. As formulated in~\cite{MRS11}: ``Since its points are coordinatized by finite binary expressions, it is the precise geometry that is represented in a digital computer. On the other hand, it provides a good model for a space with holes.''

This paper is the first of two parts
concerning the structure and properties of dyadic polygons considered as $\mc{CB}$-modes. The second part will include a characterization of isomorphism types of dyadic triangles.

The current paper is arranged as follows.
After some preliminaries on notation, Section~\ref{S:2}  provides a background necessary to understand the two parts of the paper.
It contains a summary of basic known results concerning dyadic affine spaces and dyadic convex sets in general, and dyadic polytopes in particular. Section~\ref{S:Triangles} provides an older classification of dyadic triangles.
Finally Sections~\ref{S:ClassifTriangles} and~\ref{S:5} contain the main results:  a new classification (Theorem~\ref{T:reprhats}) and a new characterization (Theorem~\ref{T:charactdyadtrs})   of pointed dyadic triangles.

All dyadic convex sets are considered in this paper as $\mathcal{CB}$-modes.
We use notation, terminology and conventions similar to
those of \cite{RS02} and the previously referenced papers. In particular, the reader may
wish to consult the papers \cite{CR13}, \cite{MMR19}, \cite{MMR19a}, \cite{MMR23} and
\cite{MRS11}. For more details and information on affine spaces, convex sets and barycentric
algebras, we also refer the reader to the monographs \cite{RS85, RS02} and the new survey
\cite{R18}. For convex polytopes, see \cite{AB83, BG03, Z95}.

\section{Dyadic convex sets}\label{S:2}

\subsection{Dyadic affine spaces}\label{S:2.1}

Recall that affine spaces over a unital subring $R$ of the ring $\mbR$ (\emph{affine $R$-spaces}) can be considered as the reducts $(A,\underline{R})$ of $R$-modules $(A,+,R)$,
where $\underline{R}$ is the set of binary affine combinations
\begin{equation}\label{E:afoper}
ab\,\underline{r} = a(1-r)+br \,
\end{equation}
for all $r \in R$ and $a,b \in A$. (See~\cite[\S5.3, \S6.3]{RS02}.) The class of all affine $R$-spaces forms a \emph{variety} of algebras (the class of all algebras satisfying a given set of identities) \cite{C75}. In this paper, the ring $R$ is the ring $\mbR$ of real numbers or the ring $\mbD$ of dyadic rational numbers. In particular,
affine spaces over the ring $\mbD$ (\emph{affine $\mbD$-spaces}) are considered here as algebras $(A,\underline{\mbD})$ with the set $\underline{\mbD} = \{\underline{d} \mid d \in \mbD\}$ of basic operations.

We are interested in affine $\mbD$-spaces from the quasivariety
$\mathsf{Q}_a(\mbD)$ (the class of all algebras satisfying a given set of quasi-identities) of affine $\mbD$-spaces, generated by the affine $\mbD$-space $\mbD$. They are all faithful (the operations $\underline{d}$ and $\underline{e}$ are different for distinct $d, e \in \mbD$). Each $n$-dimensional member of $\mathsf{Q}_a(\mbD)$ is a free affine $\mbD$-space on $n+1$ generators, and is isomorphic to the affine $\mbD$-space $\mbD^n$. (See e.g.~\cite{MMR23}.)

Note that, unlike the case of the affine $\mbR$-space $\mbR$, the affine $\mbD$-space $\mbD$
contains infinitely many pairwise distinct (though isomorphic) affine $\mbD$-subspaces.
Nontrivial subspaces of the affine $\mbD$-space $\mbD$ have the form
\[
m \mbD = \{md \mid d \in \mbD\}
\]
for some positive odd integer $m$. (Note that
$2^k (2l + 1) \mbD = (2l + 1) \mbD$ for $k,l\in\mathbb N$.)

For a finite-dimensional affine $\mbD$-space $A$ from $\mathsf{Q}_a(\mbD)$, the \emph{affine
$\mbR$-hull} $\mr{aff}_{\mbR}(A)$ of $A$ is the intersection of all affine $\mbR$-spaces
containing~$A$. The intersection $\mr{aff}_{\mbR}(A) \cap \mbD^n$ is the \emph{affine
$\mbD$-hull} $\mr{aff}_{\mbD}(A)$ of $A$.

We are mainly interested in affine dyadic spaces with a certain density property, defined as follows. An $n$-dimensional subspace $A$ of $\mbD^n$ is called \emph{geometric} precisely when it coincides with its affine $\mbD$-hull, i.e. when
\[
A = \mr{aff}_{\mbR}(A) \cap \mbD^n = \mr{aff}_{\mbD}(A).
\]
For example, while the affine $\mbD$-space $\mbD$ is geometric, its subspace $3\mbD$ is not.

Recall as well that automorphisms of the affine $\mbD$-space $\mbD^n$ form the $n$-dimensional
affine group $\mathrm{GA}(n,\mbD)$ over the ring $\mbD$, the group generated by the linear group $\mathrm{GL}(n,\mbD)$ and the group of translations of the space $\mbD^n$.   The elements of $\mathrm{GL}(n,\mbD)$ may be thought of as matrices with dyadic entries and determinants equal to $\pm 2^{k}$ for some integer $k$.

\subsection{Dyadic convex sets}

\begin{definition}\cite{CR13, MMR23}
A subgroupoid $(B, \circ)$ of the reduct $(A, \circ)$ of an affine $\mbD$-space $A$ from
$\mathsf{Q}_a(\mbD)$ (i.e. a $\circ$-subreduct $B$ of $A$) is called a \emph{dyadic convex
subset of} $A$ (or briefly a \emph{dyadic convex set}).
\end{definition}

Note that dyadic convex sets defined as above were called algebraic dyadic convex sets
in~\cite{CR13, MMR23}.

Isomorphic copies of dyadic convex sets form a (minimal) subquasivariety $\mathsf{Q}_c(\mbD)$ of the variety $\mc{CBM}$ of commutative binary modes~\cite[\S3]{MR04}.

\begin{definition}\cite{CR13, MRS11}
A dyadic convex subset of $\mbD^n$ is \emph{geometric}
if it is the intersection of a convex subset $C$ of $\mbR^n$  with the subspace~$\mbD^n$.
\end{definition}

Examples of dyadic convex sets which are not geometric are given by the $\circ$-subreducts of
the affine $\mbD$-spaces $(2m+1) \mbD$ for positive $m$.

If $C \in \mathsf{Q}_c(\mbD)$, then the \emph{affine $\mbD$-hull} $\mr{aff}_{\mbD}(C)$ of $C$ is
the intersection of all affine $\mbD$-spaces containing $C$. If $\mr{aff}_{\mbD}(C)$
is of finite dimension $n$, then we say that $C$ is
\emph{finite-dimensional}, with \emph{dimension} $\dim(C)=n$.

If $C$ is a dyadic convex subset of an affine $\mbD$-space $\mbD^n$, then the \emph{convex
$\mbR$-hull} $\mr{conv}_{\mbR}(C)$ of $C$ in $\mbR^n$ is the intersection of all convex subsets
of the affine $\mbR$-space $\mbR^n$ containing $C$. Then $\mr{conv}_{\mbR}(C)$ may be considered
as the subalgebra generated by $C$ in the real convex set $(\mbR^n, \underline{I}^\circ)$.The
\emph{convex $\mbD$-hull} ~$\mr{conv}_{\mbD}(C)$ of $C$ in $\mbD^n$ is the intersection of
$\mr{conv}_{\mbR}(C)$ with $\mbD^n$,
\begin{equation}\label{E:convhull}
\mr{conv}_{\mbD}(C) = \mr{conv}_{\mbR}(C) \cap \mbD^n,
\end{equation}
and is obviously geometric. An $n$-dimensional convex subset $C$  of an affine space
$\mbD^n$, where $n$ is a positive integer, is a geometric dyadic convex subset of $\mbD^n$
precisely if
\begin{equation}\label{E:geom}
C = \mr{conv}_{\mbD}(C).
\end{equation}

\subsection{Dyadic polytopes}\label{Ss:polytopes}

The class of dyadic convex sets we are interested in is the class of dyadic polytopes.
\begin{definition}\cite{MRS11}
A \emph{dyadic $n$-dimensional polytope} is the intersection with the dyadic space $\mbD^n$ of
an $n$-dimensional real polytope whose vertices lie in the dyadic space.
\end{definition}

Dyadic polytopes are obviously geometric, and for any dyadic $n$-dimen\-sional polytope $P$,
\begin{equation}\label{E:pol}
P = \mr{conv}_{\mbR}(P) \cap \mbD^n = \mr{conv}_{\mbD}(P).
\end{equation}
The vertices of the (real) polytope $\mr{conv}_{\mbR}(P)$ are contained in $P$, and are referred to as to the vertices of $P$. Unlike the case of real polytopes, the vertices of a dyadic polytope   (treated as a $\mathcal{CB}$-mode)  do not necessary generate the whole polytope. (See~\cite{MMR23}.)

In this paper, we are interested in dyadic polygons, especially in dyadic intervals and dyadic triangles (all considered as $\mathcal{CB}$-modes). Recall that unlike the real case, there are infinitely many pairwise non-isomorphic dyadic intervals, and infinitely many pairwise non-isomorphic dyadic triangles.

A special role is played by dyadic simplices. Dyadic simplices may be defined similarly as real simplices. An $n$-dimensional \emph{dyadic simplex} is an $n$-dimensional dyadic convex set with $n+1$ vertices, generated by these vertices. Geometric simplices are polytopes. Among dyadic polygons, all dyadic intervals isomorphic to the dyadic unit interval are $1$-dimensional simplices, and all $2$-dimensional simplices are isomorphic to the dyadic triangle generated by the three elements $e_0 = (0,0)$, $e_1 = (1,0)$ and $e_2 = (0,1)$ of $\mbD^2$.

Dyadic polytopes are considered as  $\circ$-subreducts  of their affine $\mbD$-hulls. After introducing coordinate axes in such a space, a given polytope will be located in the corresponding $\mbD$-module by providing the coordinates of its vertices. Affine $\mbD$-spaces of dimension $n$ are considered as subreducts of the $n$-dimensional real affine space $\mbR^n$. Moreover, isomorphisms of dyadic polytopes are considered as restrictions of automorphisms of their affine $\mbD$-hulls. (See~\cite{MMR23} for more explanations.)

\section{Pointed dyadic triangles}\label{S:Triangles}

While all real intervals are isomorphic as barycentric algebras, and each is generated by
its ends, there are infinitely many pairwise non-isomorphic dyadic intervals (considered as
$\mc{CB}$-modes). Dyadic intervals are characterized in~\cite{MRS11}. In particular,
each non-trivial interval of $\mbD$ is isomorphic to some dyadic interval $\mbD_k = [0,k]$,
where $k$ is an odd positive integer. Two such intervals are isomorphic precisely when their
right hand ends are equal.
For any $d \in \mbD$ and $n\in\mathbb N$, all dyadic intervals $[d,d+2^{n}]$ are isomorphic to $\mbD_1$, and the interval $\mbD_{1}$ is generated by $0$ and $1$. For $k > 1$, each dyadic interval $[d,d+k2^{n}]$ is isomorphic to $\mbD_k$, and the interval $\mbD_{k}$ is generated by three (but no fewer) elements. If an interval is isomorphic to $\mbD_k$, then we say that the interval has \emph{type} $k$. Note that among the intervals $\mbD_k$, only $\mbD_1$ is a (one-dimensional geometric) simplex.

Dyadic triangles\footnote{The dyadic triangles considered in this paper are closed (geometric) triangles contained in the plane $\mbD^2$,
which do not reduce to intervals or points. All dyadic triangles are considered as commutative
binary modes.} were described in \cite{MRS11}, with some improvements given
in \cite{MMR19a} and \cite{MMR23}. Roughly speaking, it was shown that each dyadic triangle of
the plane $\mbD^2$ is isomorphic to a triangle in the first quadrant of the plane, with one ``pointed'' vertex located at the origin and two remaining ones located as follows: either both are on two different axes, or precisely one is on an axis, or else none of them is on an axis. In particular, it was first shown in (\cite[Lemma~5.2]{MRS11}) that each dyadic triangle is isomorphic to a triangle $ABC$ contained in the first quadrant, with one pointed vertex, say $A$, located at the origin, and with the vertices $B$ and $C$ having non-negative
integer coordinates. (See Figure~\ref{F:1}.) Such triangles $ABC$ were called \emph{pointed
triangles}. The integers $i, j, m, n$ in Figure~\ref{F:1} were chosen so that $0 \leq i < m$,
$0 \leq  n < j$, and with odd $\gcd\{i,m\}$ and $\gcd\{j,n\}$. Such a triangle was denoted
$T_{i,j,m,n}$.
\begin{figure}[bht]
\begin{center}
\begin{picture}(140,120)(0,0)

\put(0,20){\vector(1,0){130}}
\put(20,0){\vector(0,1){110}}
\put(20,20){\line(1,1){60}}

\put(20,80){\line(1,0){100}}
\put(10,77){$n$}
\put(80,20){\line(0,1){80}}
\put(77,8){$i$}
\put(20,100){\line(1,0){100}}
\put(10,97){$j$}
\put(120,20){\line(0,1){80}}
\put(117,8){$m$}

\put(80,80){\circle*{3}}
\put(82,69){$G$}

\put(20,20){\circle*{5}}
\put(7,7){$A$}
\put(80,100){\circle*{5}}
\put(77,107){$B$}
\put(120,80){\circle*{5}}
\put(125,76){$C$}

\thicklines

\put(20,20){\line(5,3){100}}
\put(20,20){\line(3,4){60}}
\put(120,80){\line(-2,1){40}}

\end{picture}
\end{center}
\caption{The triangle $T_{i,j,m,n}$}
\label{F:1}
\end{figure}
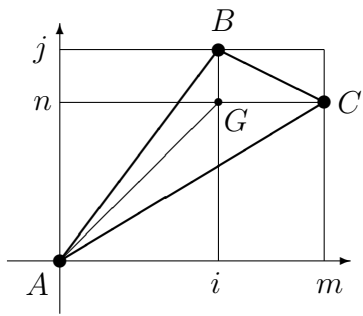

Note that each side of $T_{i,j,m,n}$ is a dyadic interval, so it is isomorphic
to some $\mbD_k$.

A triangle $T_{0,j,m,0}$ is a \emph{right triangle} (i.e. its shorter sides are parallel to the coordinate axes), and a triangle $T_{i,j,m,0}$ is a \emph{hat triangle} (i.e. one of its
sides is parallel to a coordinate axis). If the sides of a (dyadic) triangle have respective
types $r, s, t$, then the triangle has \emph{boundary type} $(r,s,t)$. Note that the type is
only defined up to cyclic order.
A right triangle is determined uniquely up
to isomorphism by its boundary type (\cite[Cor.~4.5]{MRS11}), and if the types of shorter sides are $r$ and $t$, then the type of the hypotenuse is $\gcd\{r,t\}$ (\cite[Pythagoras'
Thm.~4.3]{MRS11}). In particular, the triangle $T_{0,j,m,0}$, with odd $j$ and $m$, is
determined by its boundary type, or equivalently by the (odd) integers $j$ and $m$.

\begin{example}\label{Ex:jjm}
Consider the triangle $T_{j,j,m,0}$ with vertices $A = (0,0),$ $ B = (j,j)$ and $C = (m,0)$. Using the (invertible dyadic) matrix
\begin{equation}
\begin{bmatrix}
1 & 0 \\
-1 & 1
\end{bmatrix}\,,
\end{equation}
the triangle $ABC$ is transformed to the isomorphic right triangle\\ $T_{0,j,m,0}$ with
vertices $(0,0), (0,j)$ and $(m,0)$.
\end{example}

A hat triangle may be isomorphic to a right triangle, but \cite[Example~$4.6$]{MMR19a} exhibits hat triangles non-isomorphic to right triangles.

We recall some simple properties of automorphisms of the affine $\mbD$-plane $\mbD^2$, used frequently later on.

\begin{lemma}~\cite{MRS11}\label{L:transrefl}
All translations and all reflections of the dyadic plane $\mbD^2$ in one of the coordinate axes or in the line $y = x$ or $y = -x$ map any polygon in $\mbD^2$ onto an isomorphic polygon.
\end{lemma}

\begin{lemma}\label{L:fracint}
Let $P = ({p}/{2^r},{q}/{2^s})$ be a point of $\mbD^2$
with integers $r$, $s$ and odd integers $p $ and $q$. Then there is a $\mbD$-module
automorphism which maps $P$ to a point $P' = (p',q')$ with integral coordinates.
\end{lemma}
\begin{proof}
We use the automorphism given by the matrix
\begin{equation}\label{E:kl}
M_{kl} = \begin{bmatrix}
2^k & 0 \\
0 & 2^l
\end{bmatrix}
\end{equation}
for sufficiently large integers $k$ and $l$.
\end{proof}

\begin{remark}\label{R:fracint}
If $r \leq s$, then it suffices to take $k = l = s$, Then $P$ is mapped to $P' = (p 2^{s-r}, q)$. If $r > s$, then one can take $k = l = r$, and $P$ is mapped to $P' = (p,q 2^{r-s})$. Note that, in both cases, at least one of the coordinates of $P'$ is odd.
\end{remark}

  Similarly, one shows that if $j' = j 2^k$ for an odd $j$,  then the hat triangles
$T_{i,j,m,0}$ and $T_{i,j',m,0}$ are isomorphic. So when considering hat triangles
$T_{i,j,m,0}$, one frequently assumes that $j$ is odd.

\begin{corollary}\label{C:fracint}
Let $P_1, \dots, P_n$ be points of $\mbD^2$ with $P_i =
({p_i}/{2^{r_i}},{q_i}/{2^{s_i}})$, for $i =1, \dots, n$, with integers $r_i$,
$s_i$ and odd integers $p_i$ $q_i$. Then there is a $\mbD$-module automorphism which maps
each point $P_i$ to a point $P'_i = (p'_i,q'_i)$ with integral coordinates.
\end{corollary}
\begin{proof}
Again, one uses the matrix $M_{kl}$ for sufficiently large $k$ and $l$. In particular, one can
take $k = l = \max\{r_1, \dots, r_n, s_1, \dots, s_n\}$.
\end{proof}

Note that, if some of the $p_i$ or $q_i$ of Corollary~\ref{C:fracint} are zero, then one can use a similar method to transform all $P_i$ to $P'_i$ with integral coordinates.

Note also that, if $P_1, \dots, P_n$ are the vertices of a polygon $P$, then the $\mbD$-module automorphism of Corollary~\ref{C:fracint} maps the polygon $P$ onto an isomorphic polygon having vertices with integral coordinates.

When a vertex $A$ of a dyadic triangle located at the origin is already chosen, the triangle
$ABC$ comes in one of the three types as described in the following theorem.

\begin{theorem}\cite[\S\S5--6]{MRS11},\cite[\S1]{MMR19a}, \cite[\S6]{MMR23}\label{T:classif}
Each pointed dyadic triangle $ABC$ comes as a triangle $T_{i,j,m,n}$  in one of the following
three types:
\begin{itemize}
\item[(a)] right triangles $T_{0,j,m,0}$ with $j$ and $m$ odd and $j \leq m$;
\item[(b)]  hat triangles $T_{i,j,m,0}$ with $0 < i \leq m/2$, odd $j > 1$, and\\ $\gcd\{i,
    j\} \neq  j$;
\item[(c)] other triangles, such that neither of $i, n$ is zero, and moreover $j \leq m$,
    $\gcd\{i, j\} \notin \{i, j, 1\}$ and $\gcd\{m, n\} \notin \{m, n,1\}$.
\end{itemize}
\end{theorem}

The triangles $T_{i,j,m,n}$ described in Theorem~\ref{T:classif} were
called \emph{representative triangles} of types (a), (b) and (c), respectively.

A weak aspect of the description provided in Theorem~\ref{T:classif} is that the same triangle may
have two different types of representative triangles, e.g. for two different choices of
the point $A$, see~\cite{MMR19a}. While the original version of Theorem
~\ref{T:classif} is true, some of its consequences presented in~\cite{MRS11} are not. A
mistake in Lemma~4.11 of~\cite{MRS11} resulted in an incorrect formulation of Theorem~6.6 of
~\cite{MRS11} (repeated in~\cite{MMR19a}). This theorem says that two dyadic triangles are isomorphic if and only if they
have the same encoding quadruple $(i,j,m,n)$, i.e. the same representative triangles.

The following section contains an improved description of pointed dyadic triangles and   their  classification.

\section{Classification of dyadic triangles}\label{S:ClassifTriangles}

We start with a necessary condition for two dyadic triangles to be
isomorphic. Then, a correction to the statement of Lemma~$4.11$ in~\cite{MRS11} is presented.
Next, we show that each pointed dyadic triangle is isomorphic to a certain special triangle we call \emph{representative}. As a conclusion we obtain a full classification of   pointed   dyadic triangles. This result is then used in the final section to characterize all representative dyadic triangles.

\begin{proposition}\label{P:area}
Let $\iota:T \rightarrow T'$ be an isomorphism between two dyadic triangles $T$ and $T'$. Let
$P$ be the area of the convex $\mbR$-hull $\mr{conv}_{\mbR}(T)$ of $T$ and $P'$ be the area of
the convex $\mbR$-hull $\mr{conv}_{\mbR}(T')$ of $T'$. Then $ P/P' = 2^k$ for some integer $k$.
\end{proposition}
\begin{proof}
Let $A = (x_a,y_a)$, $B = (x_b,y_b)$, $C = (x_c,y_c)$ be the vertices of $T$. The isomorphism
$\iota$ maps the set of vertices of $T$ onto the set of vertices of $T'$. Let $A' =
(x'_a,y'_a)$, $B' = (x'_b,y'_b)$, $C' = (x'_c,y'_c)$ be the vertices of $T'$, the respective
images of $A,B$ and $C$ under $\iota$.
Then	
\[
		P = \frac{1}{2}\left| \rm{det}
		\begin{bmatrix}
			x_b-x_a&y_b-y_a\\
		    x_c-x_a&y_c-y_a
		\end{bmatrix}\right|
\]	
and
\[
		P' = \frac{1}{2} \left|\rm{det}
		\begin{bmatrix}
			x'_b-x'_a&y'_b-y'_a\\
			x'_c-x'_a&y'_c-y'_a
		\end{bmatrix}\right|.
\]
The isomorphism $\iota$ is a composition of a translation and a linear mapping given by an
invertible dyadic matrix $M$. Note that translations do not change the area of a polygon. Then

\begin{align*}
P' &= \frac{1}{2}\begin{vmatrix} \det
 \begin{bmatrix}
 	x'_b-x'_a&y'_b-y'_a\\
 	x'_c-x'_a&y'_c-y'_a
 \end{bmatrix}\end{vmatrix} \\
 &=
\frac{1}{2}\begin{vmatrix} \det\left(
\begin{bmatrix}
	x'_b &y'_b \\
	x'_c &y'_c
\end{bmatrix}- \begin{bmatrix}
 x'_a& y'_a\\
 x'_a& y'_a
\end{bmatrix}\right)\end{vmatrix}\\
&=\frac{1}{2}\begin{vmatrix} \det\left(
\begin{bmatrix}
	x_b &y_b \\
	x_c &y_c
\end{bmatrix}M- \begin{bmatrix}
	x_a& y_a\\
	x_a& y_a
\end{bmatrix}M\right)\end{vmatrix} \\
&=
\frac{1}{2}\begin{vmatrix}\det
	\left(\begin{bmatrix}
		x_b-x_a&y_b-y_a\\
	x_c-x_a&y_c-y_a
\end{bmatrix}
	 M\right)\end{vmatrix}
\\&= \frac{1}{2}\begin{vmatrix} \det
			\begin{bmatrix}
			x_b-x_a&y_b-y_a\\
			x_c-x_a&y_c-y_a
		\end{bmatrix}
	\cdot \det M\end{vmatrix} = P\cdot 2^{l}
\end{align*}
for some integer $l$.
\end{proof}

We will see later on that the converse of Proposition~\ref{P:area} does not hold.

\begin{remark}\label{R:orient}
By leaving the sign of the absolute value in the areas $P$ and $P'$, given as in the proof of Proposition~\ref{P:area}, one obtains so-called signed areas $P_s$ and $P'_s$, and one can easily see that $P'_s = P_s \cdot \det M$. The orientation of the vertices of a triangle depends on the sign of its area. It follows that if $\det M > 0$, then the orientation of the vertices of $T$ and $T'$ are the same, and if $\det M < 0$, then the isomorphism $\iota: T \rightarrow T'$ changes the orientation of the vertices.
\end{remark}

Note that isomorphic dyadic triangles with the same orientation of vertices have the same
boundary type. However, next
Example \ref{Ex:mmm0} shows that two non-isomorphic dyadic triangles may also have the same
boundary type.

\begin{example}\label{Ex:mmm0}
In~\cite[Prop.~4.7]{MRS11}, it was shown that there are infinitely many pairwise non-isomorphic
triangles of boundary type $(1,1,1)$. This observation may be extended to triangles of boundary
type $(m,m,m)$ for any odd positive integer $m$.
Let $T = ABC$ be the hat with $A = (0,0), B = (m,m)$ and $C = (2m,0)$. Let $T_k = AB_kC$ be a
hat with $B_k = (m,km)$, where $k$ is also a positive odd integer. Note that $T_1 = T$, and the
boundary type of each $T_k$ is $(m,m,m)$. By Proposition~\ref{P:area}, if $k \neq l$,
then the triangles $T_k$ and $T_l$ are not isomorphic.
\end{example}

The following lemma provides a correction to the statement of Lemma\\$4.11$ in~\cite{MRS11}.

\begin{lemma}\label{L:onaxis}
Let $P$ be a point of the plane $\mbD^2$ with integral coordinates, and not belonging to the
coordinate axes. There is a $\mbD$-module automorhism which takes the point $P$ onto a point of one of the coordinate axes, again with integral coordinates.
\end{lemma}
\begin{proof}
Let $P = (ma,mb)$ with (positive) $m$ and such that $\gcd\{a,b\} = 1$. Let $M_{ax}$ be the dyadic
matrix
\begin{equation}\label{E:Max}
M_{ax} =
\begin{bmatrix}
y & -b \\
-x & a
\end{bmatrix}
\end{equation}
with integers $x$ and $y$, and such that $ \det (M_{ax}) = ay - bx = 1$. (Note that the latter
equation has an integral solution.) Then $(ma,mb) M_{ax} = (may - mbx, -mab + mab) = (m,0)$.

In a similar way (using the matrix $M_{ax}$ with changed order of the columns), one can map the point $P$ to the point $(0,m)$ of the $y$-axis.
\end{proof}
\noindent  Note that using the reflexion
in the line $y = x$, one can always transform $(0,m)$ to $(m,0)$ and $(m,0)$ to $(0,m)$.

Any dyadic triangle $T$ considered in this section will be thought of as a triangle with
clockwise ordered vertices $A, B, C$, and will be denoted by $ABC$. Note that any vertex may be choosen to be in the first place.
In all transformations of $T$, each of the three vertices will be mapped to a
vertex denoted by the same letter, but possibly with a different index.

\begin{proposition}\label{P:newclassif1}
Let $T$ be a dyadic triangle in the dyadic plane $\mbD^2$. Then $T$ is isomorphic
to each of the following triangles with vertices having integral coordinates:
\begin{enumerate}
\item[(a)] A triangle $ABC$ contained in the upper halfplane of $\mbD^2$ with one vertex, say $A$, located at the origin, and another vertex, say $C = (m,0)$, on the positive part of the $x$-axis;
\item[(b)] A triangle $ABC$ contained in the first quadrant of $\mbD^2$ with $A$ and $C$
    defined as above, and the third vertex $B = (i,j)$ such that $0 \leq i \leq m$.
\end{enumerate}
\end{proposition}
\begin{proof}
(a) Let $T$ be a triangle $A_0B_0C_0$ with clockwise ordered vertices.
First note that $T$ may be mapped to an isomorphic triangle with one vertex located at the
origin. Just choose one vertex, say $A_0$, and translate $T$ to a triangle $A_1B_1C_1$ with the image $A = A_1$ of $A_0$ located at the origin.

Then transform $AB_1C_1$ to the triangle $AB_2C_2$ with vertices having integral coordinates,
using Corollary~\ref{C:fracint}.

Now suppose that the vertex $C_2$ with integral coordinates is not already on the $x$-axis, and write it in the form $C_2 = (ma,mb)$, where $\gcd\{a,b\} = 1$, as described in the proof of Lemma~\ref{L:onaxis}. The triangle $AB_2C_2$ will be mapped to the isomorphic triangle $AB_3C_3$  using the matrix $M_{ax}$ of Lemma~\ref{L:onaxis}. The coordinates of $B_3$ and $C_3$ will be still integral, with $C_3 = (m,0)$ on the positive part of the $x$-axis.  Note as well that since $ \det (M_{ax}) = 1 > 0$, it follows that the orientation of the vertices of the triangles $AB_2C_2$ and $AB_3C_3$ must be the same. (Cf. Remark~\ref{R:orient}.) So $B_3$ is contained in the upper halfplane, and one obtains a triangle $ABC$ in the upper halfplane with $B = B_3$ and $C = C_3$.

(b) Consider the triangle $AB_3C$ with $C = C_3$ obtained in the first part of the proof. If the first coordinate of $B_3 = (i',j')$ is negative, then consider the dyadic plane automorphism $\iota$ given by the dyadic matrix
\begin{equation}\label{E:M}
\begin{bmatrix}
1 & 0 \\
c & 1
\end{bmatrix}
.
\end{equation}
The automorphism $\iota$ keeps the vertices $A$ and $C$ fixed, and maps $B_3$ to $B_4 = (i'+j'c,j')$. Then $0 \leq i'+j'c \leq m$ precisely when $-i'/j' \leq c \leq (m-i')/j'$.
Note that for positive $m$, such a dyadic number $c$ always exists.

One more adjustment may be needed in the case when $i'+j'c$ is not an integer, say $i'+j'c =
n/2^{k}$ for some integers $n$ and $k$. But then, one can use the matrix~\eqref{E:kl}, with $l = 0$, to map the triangle $AB_4C$ to the isomorphic triangle $ABC = AB_5C_5$, where $B_5 = (i,j) = (n,j')$ and $C_5 = (m',0) = (m2^k,0)$ with the integer $i$ satisfying $0 \leq i \leq m'$.
\end{proof}

\begin{remark}\label{R:b2c2}
If, in the third paragraph of the proof of Proposition~\ref{P:newclassif1}, one starts with the vertex $B_2$ instead of $C_2$ and one puts $B_2 = (ma,mb)$ with $\gcd\{a,b\} = 1$ and a positive integer $m$, then the triangle $AB_2C_2$ is mapped onto the isomorphic triangle $AC_3B_3$ with $B_3 = (m,0)$ on the positive part of the $x$-axis, and $C_3$ with integral coordinates. (Note the change of order of vertices.) Then one proceeds as in the proof of
Proposition~\ref{P:newclassif1} to obtain an isomorphic triangle $ACB$ with the required properties.
\end{remark}

\begin{remark}\label{R:rightshats}
Consider a triangle $ABC$ as constructed in Proposition~\ref{P:newclassif1} with $A = (0,0), B
=(i,j), C = (m,0)$.

If $i = m$, then using Lemma~\ref{L:transrefl}, one can transform $ABC$ to the isomorphic
triangle $T_{0,j,m,0}$. If $j > m$, then using Lemma~\ref{L:transrefl} again, one can transform $ABC$ onto the triangle $T_{0,m,j,0}$ with $m < j$. By Lemma~\ref{L:fracint}, one may assume that both $j$ and $m$ are odd. So up to isomorphism, right triangles may be considered as the triangles $T_{0,j,m,0}$ with $j$ and $m$ odd and $j \leq m$.

If $0 < i < m$, the triangle $ABC$ is a hat triangle. Using an appropriate matrix~\eqref{E:kl}
one can transform the triangle $ABC$ onto the isomorphic hat triangle $AB'C'$ with $B' =
(i',j')$ and $C' = (m',0)$ such that $j'$ and $\mr{gcd}\{i',m'\}$ are odd. So, up to
isomorphism, one can consider hat triangles as triangles $T_{i,j,m,0}$ with odd $j$ and odd
$\mr{gcd}\{i,m\}$.

Finally note that a similar transformation may be used in the case of a triangle with the
coefficient $i$ smaller that $0$ or bigger than $m$ to obtain the triangle $AB'C'$ with odd $j'$ and odd $\mr{gcd}\{i',m'\}$. So also such a triangle may be considered up to isomorphism as a triangle with vertices $A = (0,0), B = (i,j)$ and $C = (m,0)$ with odd $j$ and odd
$\mr{gcd}\{i,m\}$.
However, such a triangle does not need to have the form of a hat $T_{i,j,m,0}$ with non-negative $i,j,m$, as defined above.
\end{remark}

Taking into account the last observation, we extend our notation of $T_{i,j,m,0}$ to the case
when $i$ is any integer. Note that also in this more general case, $T_{i,j,m,0}$ may be
considered (up to isomorphism) as a triangle with odd $j$ and odd $\mr{gcd}\{i,m\}$.

There is one more convenient possibility to describe coordinates of the vertices of hat
triangles, provided by the next proposition.
But first consider the following example and lemma.

\begin{example}\label{Ex:i1m}
The matrix
\begin{equation}
\begin{bmatrix}
1 & 0 \\
-i & 1
\end{bmatrix}
,
\end{equation}
transforms each triangle $T_{i,1,m,0}$ with vertices $A = (0,0), B = (i,1)$ and $C =
(m,0)$ onto the isomorphic right triangle $AB'C$ with vertices $A, B'= (0,1)$ and $C$.
If $m = m'2^k$, with odd $m'$, then $AB'C$ is isomorphic to $AB'C'$ with $C' = (m',0)$
\end{example}

\begin{lemma}\label{L:infisoms}
Consider a hat triangle $T_{i,j,m,0}$. Then each hat triangle $T_{l,j,m,0}$,  with $l = i + jk$ and $k \in \mbZ$, is isomorphic to $T_{i,j,m,0}$.
\end{lemma}
\begin{proof}
The automorphism of the space $\mbD^2$ given by the matrix
\begin{equation}\label{E:Mk}
M_k =
\begin{bmatrix}
1 & 0 \\
k & 1
\end{bmatrix},
\end{equation}
where $k \in \mbZ$,
maps the triangle $T_{i,j,m,0}$ onto the isomorphic triangle $T_{i+jk,j,m,0}$.
\end{proof}
Note that by Lemma~\ref{L:infisoms}, there are infinitely many triangles $T_{l,j,m,0}$
isomorphic to $T_{i,j,m,0}$.

\begin{remark}\label{R:infisoms}
By Lemma~\ref{L:transrefl}, the triangles $T_{i,j,m,0}$ and $T_{m-i,j,m,0}$ are
isomorphic. Indeed, if $T = ABC$ is the triangle $T_{i,j,m,0}$ with $A = (0,0), B = (i,j)$ and
$C = (m,0)$, then using the reflection in the $y$-axis and then the translation by the vector
$(m,0)$, one obtains the isomorphic triangle $T' = C'B'A'$, with $C' = (0,0), B' = (m-i,j)$ and
$A' = (m,0)$, which is the triangle $T_{m-i,j,m,0}$. Note that the isomorphism $T$ onto $T'$
reverses the order of vertices. If the boundary type of $T$ is $(r,s,t)$, then the boundary type of $T'$ is $(t,s,r)$.

Figure~\ref{F:2} illustrates three possible localizations $B_1, B_2, B_3$ of the vertex $B$ and $B'_1, B'_2, B'_3$ of the vertex $B'$, with the first respective coordinate negative, contained between $0$ and $m$, and bigger than $m$.

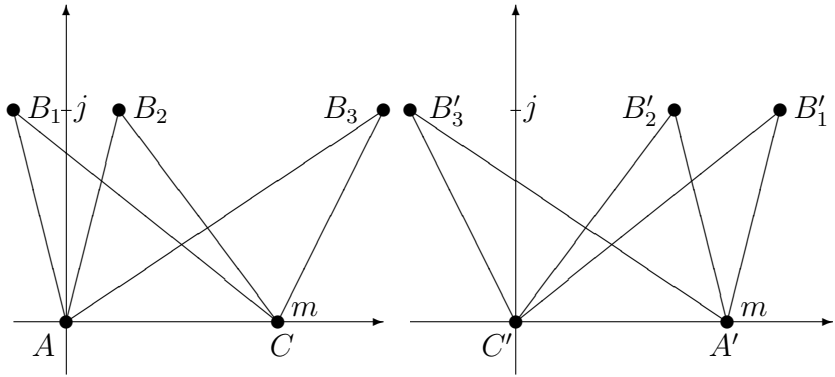
\begin{figure}[H]
	\begin{center}
		\begin{picture}(300,140)
			\put( 0,20){\vector(1,0){140}}
			\put(20,0){\vector(0,1){140}}
			\put(20,20){\circle*{5}}
			
			\put(100,20){\circle*{5}}
			\put(40,100){\circle*{5}}
			\put(20,20){\line(1,4){20}}
			\put(100,20){\line(-3,4){60}}
			\put(18,100){\line(1,0){4}}
			\put(140,100){\circle*{5}}
			\put(20,20){\line(3,2){120}}
			\put(140,100){\line(-1,-2){40}}
			\put(0,100){\circle*{5}}
			\put(20,20){\line(-1,4){20}}
			\put(100,20){\line(-5,4){100}}
			\put(7,7){$A$}
			\put(45,97){$B_2 $}
			\put(97,7){$C $}
			\put(105,22){$m$}
			\put(117,97){$B_3$}
			\put(5,97){$B_1$}
			\put(23,98){$j$}
			\put(150,20){\vector(1,0){160}}
			\put(190,0){\vector(0,1){140}}
			\put(190,20){\circle*{5}}
			\put(270,20){\circle*{5}}
			
			\put(250,100){\circle*{5}}
			\put(190,20){\line(3,4){60}}
			\put(270,20){\line(-1,4){20}}
			\put(188,100){\line(1,0){4}}
			\put(150,100){\circle*{5}}
			\put(190,20){\line(-1,2){40}}
			\put(150,100){\line(3,-2){120}}
			\put(290,100){\circle*{5}}
			\put(270,20){\line(1,4){20}}
			\put(190,20){\line(5,4){100}}
			\put(193,98){$j$}
			
			\put(177,7){$C'$}
			\put(275,22){$m$}
			\put(263,7){$A'$}
			\put(295,97){$B_1'$}
			\put(230,97){$B_2 '$}
			
			\put(157,97){$B_3'$}

		\end{picture}
	\end{center}
	\caption{The isomorphic hats $T_{i,j,m,0}$ and $T_{m-i,j,m,0}$.}
	\label{F:2}
\end{figure}

\bigskip
By Lemma~\ref{L:infisoms}, there are infinitely many triangles $T_{l,j,m,0}$ isomorphic to
$T_{m-i,j,m,0}$ with integral coordinates $l = m-i + jk, j, m$ for any integer $k$.
\end{remark}

\begin{proposition}\label{P:oddcoords}
Each triangle $T_{i,j,m,0}$ with odd $j$ and odd $\mr{gcd}\{i,m\}$
is isomorphic to a triangle $T_{i',j,m',0}$, where all $i', j, m'$ are
odd.
\end{proposition}
\begin{proof}
Assume that $j$ and $\mr{gcd}\{i,m\}$ are odd.
The case $j = 1$ was considered in Example~\ref{Ex:i1m}. Now let $j > 1$.
One has to consider two cases, when precisely one of $i$ and $m$  is odd.

Case A. Assume that $i$ is even and $m$ is odd. Then the automorphism of the space $\mbD^2$
given by the matrix $M_1$ of Lemma~\ref{L:infisoms} maps the triangle $T_{i,j,m,0}$ onto the
isomorphic triangle $T_{i+j,j,m,0}$ with odd $i+j$.

Case B. Now assume that $m$ is even and $i$ is odd. The proof of the lemma goes by induction on $m$. For $m = 2$, the automorphism of the space $\mbD^2$ given by the matrix
\begin{equation}\label{E:N}
N =
\begin{bmatrix}
1/2 & 0 \\
1/2 & 1
\end{bmatrix}
\end{equation}
maps the triangle $T_{i,j,2,0}$ onto the isomorphic triangle $T_{i',j,1,0}$, where $i' =
(i+j)/2$. If $i'$ is even, then the automorphism given by the matrix $M_1$ will transform $(i',j)$ to $(i'',j)$ with odd $i'' = i' + j$ and keep the other vertices the same.

Now assume that the proposition holds for all $m < n$, and let $m = n$. The automorphism of
$\mbD^2$ given by the matrix $N$ maps the vertex $(m,0)$ to $(m/2,0)$, with $m/2 < n$,
and $(i,j)$ to $(i',j)$, where $i' = (i+j)/2$.

If $i+j = m$, then $i' = m/2$, and one obtains the right triangle $T_{m/2,j,m/2.0}$. If $m/2 =
m' 2^k$ with odd $m'$, then $T_{m/2,j,m/2,0}$ is isomorphic to $T_{m',j,m',0}$, and the last one to $T_{0,m',j,0}$ with odd $j$ and $m'$.

Assume that $i+j \neq m$. If $i'$ is even, then we use the matrix $M_1$ again to obtain the triangle with vertices $(0,0), (m/2,0)$ and $(i'+j,j)$ with odd $i'+j$. Since $m/2 < n$, the induction hypothesis implies that the last triangle is isomorphic to a triangle with the vertices having all non-zero coordinates odd.
\end{proof}

\begin{remark}\label{R:oddcoords}
Note that if all $i, j, m$ in $T_{i,j,m,0}$ are odd, then $j$ and $\mr{gcd}\{i,m\}$ are always
odd. On the other hand, all triangles $T_{i-kj,j,m,0}$, with odd $k$, are isomorphic to
$T_{i,j,m,0}$ with even $i-kj$. Finally note also that one can find a triangle $T_{i',j,m',0}$
with odd $\mr{gcd}\{i',m'\}$ and $0 \leq i' \leq m'$ isomorphic to $T_{i,j,m,0}$. Let
$T_{i,j,m,0}$ be a triangle $T = ABC$ with all $i, j, m$ odd. Then similarly as in the proof of Proposition~\ref{P:newclassif1}, we use the matrix~\eqref{E:M} to transform the triangle $T$ to the isomorphic triangle $AB'C$ with $B' = (i+jc,j)$ with $0 \leq i+jc \leq
m$. If $i+jc = n/2^k$ (with odd $n$) is not an integer, then we use the matrix $M_{k,0}$ to
transform $AB'C$ onto the isomorphic triangle $AB''C'$ with $B'' = (n,j)$ and $C' = (2^k m,0)$. Obviously $j$ and $\gcd\{n,2^k m\}$ are odd integers and $0 \leq n \leq 2^k m$.
\end{remark}

In what follows we use the name \emph{generalised hats} or briefly \emph{hats} for dyadic right or hat triangles $T_{i,j,m,0}$ with any integer $i$ and positive integers $j, m$, and sometimes \emph{right hats} in the case when $i = 0$ or $i = m$, \emph{proper hats} if $0 < i < m$ and \emph{crooked hats} in remaining cases.

A triangle $T_{i,j,m,0}$ will be called a \emph{representative hat} when all three $i, j$ and
$m$ are odd integers.
A representative hat $T_{i,j,m,0}$ will be denoted by $T_{i,j,m}$.

The following theorem is a direct consequence of Proposition~\ref{P:newclassif1},
Remark~\ref{R:rightshats} and Proposition~\ref{P:oddcoords}.

\begin{theorem}\label{T:reprhats}
Each dyadic triangle in the dyadic space $\mbD^2$ is isomorphic to a representative hat.
\end{theorem}

\begin{remark}\label{R:6reprhats}
Proposition~\ref{P:newclassif1} implies that each dyadic triangle $ABC$ is isomorphic to three
representative hats with the same ordering of the vertices and three isomorphic hats with the
reverse ordering. They correspond to the six permutations of the vertices.
\end{remark}

\begin{remark}\label{R:rightrepr}
Note that the triangles $T_{i,j,m,0}$ and $T_{0,m,j,i}$ are isomorphic. Note, as well, the
following isomorphisms of right triangles: $T_{0,j,m,0}$ and $T_{0,m,j,0}$, then $T_{m,j,m}$ and $T_{0,j,m,0}$ and finally $T_{j,j,m}$ and $T_{0,j,m,0}$, where $j$ and $m$ are odd positive integers. Sometimes it is convenient to consider right triangles up to isomorphisms as triangles $T_{0,j,m,0}$ with $j, m$ odd and $j \leq m$, and denote them by $T_{j,m}$.
\end{remark}

\section{Characterization of pointed dyadic triangles}\label{S:5}

The next question of interest is when representative hats are isomorphic. In this section
we will consider representative hats as pointed triangles with a pointed vertex located at the
origin, and isomorphisms of representative hats as isomorphisms preserving the pointed vertex
and the orientation of vertices. This type of an isomorphism will be called a \emph{pointed
oriented isomorphism} or briefly a \emph{pointed isomorphism}. First note that by
Lemma~\ref{L:infisoms}, each representative hat $T_{i,j,m}$ is isomorphic to infinitely many
hats $T_{l,j,m,0}$ with $l = i + jk$, where $k$ is an integer.

The following proposition describes the class of representative hats pointed isomorphic to a given representative hat.

\begin{proposition}\label{P:hathat}
Consider two representative hats $T = ABC$ and $T' = A'B'C'$ with the vertices $A = A' = (0,0)$, $B = (i,j), C = (m,0)$ and $B' = (i',j'), C' = (m',0)$. Let $\iota: T \rightarrow T'$ be a mapping taking $B$ to $B'$ and $C$ to $C'$.

Then $\iota: T \rightarrow T'$ is a pointed oriented isomorphism precisely when $j' = j$, $m' = m$ and $i' = i + jk$ for some even integer $k$.
 \end{proposition}
\begin{proof}
First assume that $\iota$ is a pointed isomorphism.
Since any isomorphism between two triangles preserves the side types and both $T$ and $T'$  are representative, it follows immediately that $m = m'$. Then by Proposition~\ref{P:area}, the areas $P$ of the convex $\mbR$-hull $\mr{conv}_{\mbR}(T)$ of $T$ and $P'$ of the convex
$\mbR$-hull $\mr{conv}_{\mbR}(T')$ of $T'$ are equal. Hence $mj = mj'$, and consequently $j =
j'$.

The mapping $\iota$ is an isomorphism preserving the pointed vertex and the orientation
precisely when there exists an invertible dyadic matrix $M$ such that
\begin{equation}
\begin{bmatrix}
m&0\\
i&j
\end{bmatrix}
M =
\begin{bmatrix}
m'&0\\
i'&j'
\end{bmatrix}
=
\begin{bmatrix}
m&0\\
i'&j
\end{bmatrix}
.
\end{equation}

It follows that
\begin{equation}
M = \frac{1}{mj}
\begin{bmatrix}
	j&0\\
	-i&m
\end{bmatrix}
\begin{bmatrix}
	m&0\\
	i'&j
\end{bmatrix}
=
\begin{bmatrix}
	1&0\\
	(-i+i')/j&1
\end{bmatrix}
.\end{equation}

The number $k := (i' - i)/j$ is a dyadic number precisely when $j$ divides $i' - i$. Since $i,
i', j$ are integers, it follows that $k$ is an integer too. And since $i'$ is odd, $k$ must be
even. Then $i' - i = jk$ for some even integer $k$.

Now assume that $j = j'$, $m = m'$ and $i' - i = jk$ for some even integer $k$. Then the matrix $M$ provides a pointed isomorphism from $T$ to $T'$.
\end{proof}

Let us note that the areas of the convex $\mbR$-hulls of all triangles $AB''C$, where $B''
= (i'',j)$ and $i''$ is any integer, are all equal to $mj/2$. However, not all these triangles
are pairwise isomorphic.

\begin{remark}\label{R:hathat}
Consider the representative hat $T_{i,j,m}$. Recall that the representative hat $T_{i,j,m} =
T_{i,j,m,0}$ and the hat $T_{m-i,j,m,0}$ are isomorphic. (Cf.~Remark~\ref{R:infisoms}. Recall that if the hat $T_{i,j,m}$ is $ABC$, then the hat $T_{m-i,j,m,0}$ is $C'B'A'$.) Since
$m-i$ is even, $T_{m-i,j,m,0}$ is not representative. Using the matrix $M_k$ of
Lemma~\ref{L:infisoms} with an odd $k$, one easily shows that $T_{m-i,j,m,0}$ is isomorphic to
the representative hat $T_{m-i+kj,j,m}$ with odd $m-i+kj$. In particular, $T_{m-i,j,m,0}$ is
isomorphic to the representative hat $T_{m-i+j,j,m}$. Applying Proposition~\ref{P:hathat} one
shows that $T_{m-i+kj,j,m}$ and $T_{l,j',m'}$ are isomorphic (with the pointed vertex and the
orientation preserved) precisely when $j' = j$, $m' = m$ and $l = (m-i+kj) + jk'$ for some even integer $k'$ or equivalently $l = m-i + jk''$ for some odd integer $k''$. Finally note that if $T_{i,j,m}$ is $ABC$, $T_{m-i,j,m,0}$ is $C'B'A'$, then $T_{m-i+jk'',j,m}$ is $C''B''A''$.
\end{remark}

\begin{corollary}\label{C:nonisomhats}
For any positive odd integers $j$ and $m$, there are $j$  (pointed) isomorphism classes of
representative hats $T_{i,j,m}$. Each class is represented by a unique $T_{i,j,m}$, where $i \in
\{1,3,\dots, 2j-1\}$.
\end{corollary}

If $i'$ in Proposition~\ref{P:hathat} equals $0$, then the hat $AB'C$ is a right hat. In this
case, one obtains the following corollary.

\begin{corollary}\label{R:hatright}
Let $T$ be a hat as in Proposition~\ref{P:hathat}. Then
$T$ is pointed isomorphic to the right hat $T'$ precisely when
$j$ divides $i$.
\end{corollary}

\begin{example}\label{Ex:isohats}
Consider two representative hats $T_{1,3,19}$ and $T_{7,3,19}$ of the same boundary type
$(1,3,19)$. Since $3$ divides $7-1 = 6$, the triangles are pointed isomorphic.
Now the hat $T_{1,3,19}$ is isomorphic (though not pointed isomorphic) to $T_{19-1,3,19,0}$, and to the representative hat $T_{21,3,19}$.
Since $3$ divides $21$, they both are isomorphic to the right triangle $T_{0,3,19,0}=T_{3,19}$.
\end{example}

\begin{example}\label{Ex:nisohats}
On the other hand, the representative hats $T_{1,5,29}$ and $T_{7,5,29}$ of the same
boundary type $(1,1,29)$ are not pointed isomorphic, since $5$ does not divides $7-1 = 6$.
\end{example}

The final part of this section is devoted to a characterization of pointed dyadic triangles.

For a dyadic triangle $T$ with clockwise ordered vertices and a pointed vertex, let $\mcR(T)$ be the class of all representative hats isomorphic to $T$ with the same pointed vertex and the same orientation of vertices.
By Corollary~\ref{C:nonisomhats}, each member of $\mcR(T)$ is isomorphic to a representative hat $T_{i,j,m}$, where
$j$ and $m$ are positive odd integers, and $i$ belongs to $\{1, 3, \dots, 2j-1\}$.
On the other hand, if $(i,j,m)$ is a triple of odd integers, where $j$ and $m$ are odd and $i \in \{1, 3, \dots, 2j-1\}$, then clearly it determines the representative hat $T_{i,j,m}$ and the class of hats $T_{i+jk,j,m}$, with even $k$, (pointed) isomorphic to $T_{i,j,m}$. We call such a triple an \emph{encoding triple}.

The following theorem characterises pointed dyadic triangles, and provides a correction to the statement of Theorem~$6.6$ of ~\cite{MRS11}.

\begin{theorem}\label{T:charactdyadtrs}
Two pointed dyadic triangles with the same orientation of vertices are isomorphic if and only if they have the same encoding triples.
\end{theorem}


\begin{thebibliography}{99}

\bibitem{B05}
G. Bergmann,
On lattices of convex sets in $\mbR^n$,
{\em Algebra Unversalis} {\bf53} (2005) 357--395.\\ \texttt{ https://doi.org/10.1007/s00012-005-1934-0}

\bibitem{AB83}
A. Br\o ndsted,
    {\em An Introduction to Convex Polytopes,}
    Springer Verlag, New York, 1983.\\
\texttt{https://doi.org/10.1007/978-1-4612-1148-8}


\bibitem{C75}
B. Cs\'{a}k\'{a}ny,
Varieties of affine modules,
{\em Acta Sci. Math.} {\bf 37} (1975) 3--10.


\bibitem{CR13}
G. Cz\'{e}dli, A. Romanowska,
Generalized convexity and closure conditions,
{\em Int. J. of Algebra and Computation} {\bf 23} (2013) 1805--1835.
\\ \texttt{https://doi.org/10.1142/S0218196713500458}

\bibitem{BG03}
B. Gr\"{u}nbaum,
    {\em Convex Polytopes,}
    2nd ed., Springer Verlag, New York, 2003.
\\ \texttt{ https://doi.org/10.1007/978-1-4613-0019-9 }


\bibitem{GH}
G. Hamm,
Classification of lattice triangles by their two smallest widths, preprint 2023, arXiv: 2304.03007.
	

\bibitem{JK75}
J. Je\v{z}ek and T. Kepka,
The lattice of varieties of commutative idempotent abelian distributive groupoids,
{\em Algebra Universalis} {\bf 5} (1975) 225--237.
\\ \texttt{https://doi.org/10.1007/BF02485256  }

\bibitem{JK76}
J. Je\v{z}ek and T. Kepka,
Free commutative idempotent abelian groupoids and quasigroups,
{\em Acta Univ. Carolin. Math. Phys.} {\bf 17} (1976) 13--19.\\
\texttt{http://dml.cz/dmlcz/142385}

\bibitem{JK83}
J. Je\v{z}ek and T. Kepka,
{\em Medial Groupoids,}
Academia, Praha, 1983.


\bibitem{MMR19}
K. Matczak, A. Mu\'{c}ka, A. Romanowska,
Duality for dyadic intervals,
{\em Int. J. of Algebra and Computation},  {\bf 29} (2019) 41--60.
\\ \texttt{ https://doi.org/10.1142/S0218196718500625 }

\bibitem{MMR19a}
K. Matczak, A. Mu\'{c}ka, A. Romanowska,
Duality for dyadic triangles,
{\em Int. J. of Algebra and Computation}, {\bf 29} (2019) 61--83.
\\ \texttt{ https://doi.org/10.1142/S0218196718500637 }

\bibitem{MMR23}
K. Matczak, A. Mu\'{c}ka, A. Romanowska,
Finitely generated dyadic convex sets,
{\em Int. J. of Algebra and Computation}, {\bf 33} (2023) 585--615.
\\ \texttt{ https://doi.org/10.1142/S0218196723500273 }

\bibitem{MR04}
K. Matczak, A. Romanowska,
    Quasivarieties of cancellative commutative binary modes,
     {\em Studia Logica} {\bf 78} (2004) 321--335.
\\ \texttt{ https://doi.org/10.1007/s11225-005-1335-6 }

\bibitem{MRS11}
K. Matczak, A. Romanowska, J.~D.~H. Smith,
   Dyadic polygons,
   {\em Int. J. of Algebra and Computation} {\bf 21} (2011) 387--408.
\\ \texttt{ https://doi.org/10.1142/S0218196711006248 }

\bibitem{MR16}
A. Mu\'{c}ka, A. Romanowska,
Duality for quasi-polytopes,
{\em Journ. of the Australian Math. Society} {\bf 101} (2016), 95--117.
\\ \texttt{ https://doi.org/10.1017/S1446788715000683  }

\bibitem{R18}
A. B. Romanowska,
Convex sets and barycentric algebras,
{\em in: Nonassociative mathematics and its applications}, 243-–259, {\em Contemp. Math.} {\bf 721},
Amer. Math. Soc., Providence, RI, 2019.
\texttt{https://doi.org/10.1090/conm/721/14509}

\bibitem{RS85}
A. B. Romanowska, J. D. H. Smith,
    {\em Modal Theory,}
    Heldermann Verlag, Berlin, 1985.

\bibitem{RS02}
A. B. Romanowska, J. D. H. Smith,
    {\em Modes,}
    World Scientific, Singapore, 2002.
\\ \texttt{https://doi.org/10.1142/4953}


\bibitem{PRS}
P. R. Scott
On convex lattice polygons.
{\em Bull. Austral. Math. Soc.} {\bf 15} (1978) 395--399.
\\ \texttt{https://doi.org/10.1017/S0004972700022826 }	
	


\bibitem{Z95}
G. M. Ziegler,
{\em Lectures in Polytopes},
Springer-Verlag, New York, 1995.
\\ \texttt{ https://doi.org/10.1007/978-1-4613-8431-1 }
\end{thebibliography}
\end{document}